\documentclass[12pt]{amsart} 
\advance\hoffset by -0.9 truecm

\usepackage{amsfonts, amsmath, amssymb, mathrsfs, mathtools}
\usepackage{latexsym, stmaryrd, array, hyperref, xcolor}
\usepackage{txfonts}

\usepackage{enumitem}
\setlist[enumerate]{font={\rm},itemsep=0.2\baselineskip}
\setlist[enumerate,1]{label={(\roman*)}}
\setlist[enumerate,2]{label={(\arabic*)}}


\newtheorem{theorem}{Theorem}[section]

\newtheorem{proposition}[theorem]{Proposition}

\newtheorem{corollary}[theorem]{Corollary}

\newtheorem{lemma}[theorem]{Lemma}

\theoremstyle{definition}

\newtheorem{construction}[theorem]{Construction}

\newtheorem{problem}[theorem]{Problem}

\theoremstyle{remark}

\newtheorem{example}[theorem]{Example}

\newtheorem*{conj*}{Conjecture}

\def\GL{\mathrm{GL}}

\def\S{\mathrm{S}}
\def\A{\mathrm{A}}
\def\D{\mathrm{D}}

\def\Sym{\mathrm{Sym}}
\def\Alt{\mathrm{Alt}}
\def\Cos{{\rm Cos}}
\def\Aut{{\rm Aut}}

\def\leqs{\leqslant}
\def\geqs{\geqslant}

\def\val{\mathrm{val}}

\def\K{{\bf K}}
\def\N{{\bf N}}

\begin{document}

\title{Symmetric Covers and Pseudocovers of Complete Graphs}

\author{Yan Zhou Zhu}
\address{Department of Mathematics, Southern University of Science and Technology\\
Shenzhen 518055, Guangdong\\
P. R. China}
\email{zhuyz2019@mail.sustech.edu.cn}

%\subjclass[]{}

\begin{abstract}
	We first characterize all faithful arc-transitive covers of complete graphs and we give a general construction of such covers.

	For a counterpart of cover, we say a graph $\Gamma$ is a pseudocover of its quotient $\Sigma$ if they have the same valency and $\Gamma$ is not a cover of $\Sigma$.
	As the second result of this paper, we prove that the complete graph $\K_n$ has a connected arc-transitive pseudocover if and only if $n-1$ is not a prime.
\end{abstract}

\maketitle

\section{Introduction}\label{sec:introduction}

All graphs considered in this paper are finite, undirected and simple.
An \textit{arc} of a graph is an ordered pair of adjacent vertices, and a $2$-arc is a triple of distinct consecutive vertices.
A graph $\Gamma$ is said to be \textit{$G$-arc-transitive} or \textit{$(G,2)$-arc-transitive} if $G\leqs\Aut\Gamma$ is transitive on the arcs or $2$-arcs of $\Gamma$ respectively.
An arc-transitive graph is usually called a \textit{symmetric graph}.
The study of symmetric graphs is one of the main topics in algebraic graph theory, refer to \cite{biggs1974Algebraic,godsil2001Algebraic} for references.

For a $G$-arc-transitive graph $\Gamma=(V,E)$, the group $G\leqslant\Aut\Gamma$ is a transitive permutation group on the vertex set $V$.
A \textit{block system} of $G$ acting on the vertex set $V$ is a \textit{non-trivial $G$-invariant partition} $\mathcal{B}$ of $V$ where $B^g\in \mathcal{B}$ for any $B\in\mathcal{B}$ and $g\in G$.
If there is no non-trivial block system $\mathcal{B}$ admitted by $G$, namely each block system has size equal to $1$ or $|V|$, then we say that $G$ is \textit{primitive} on $V$ and $\Gamma$ is a \textit{$G$-primitive} symmetric graph.
The well-known O'Nan-Scott Theorem~\cite{liebeck1988NanScott} of permutation group theory and the theory of finite simple groups provide powerful tools for the study of symmetric graphs which are vertex-primitive, see \cite{li2001finite,praeger1993NanScott}.

Assume that $G$ is \textit{imprimitive} on $V$ with a non-trivial block system $\mathcal{B}$, then $G$ induces a transitive permutation group $G^\mathcal{B}$ on $\mathcal{B}$.
The \textit{quotient graph} $\Gamma_\mathcal{B}$ of $\Gamma$ is defined to be the graph with vertex set $\mathcal{B}$ and two blocks $B,C\in\mathcal{B}$ are adjacent in $\Gamma_\mathcal{B}$ if and only if there exist adjacent vertices $\beta$ and $\gamma$ in $\Gamma$ for some $\beta\in B$ and $\gamma\in C$.
We usually call the original graph an {\it extender} of the quotient graph.

Let $\Gamma$ be a $G$-arc-transitive graph and $\Gamma_\mathcal{B}$ be a quotient of $\Gamma$.
For an arc $(B,C)$ in $\Gamma_\mathcal{B}$, the subgraph $\Gamma[B,C]$ of $\Gamma$ generated by vertices in $B\cup C$ is a bipartite graph.
Since $\Gamma$ is $G$-arc-transitive, the subgraph $\Gamma[B,C]$ is independent of the choice of arc $(B,C)$ in $\Gamma_\mathcal{B}$.
In the particular case where $\Gamma[B,C]$ is a \textit{perfect mathching} between $B$ and $C$, we say $\Gamma$ is a \textit{cover} of $\Gamma_\mathcal{B}$.
Constructing and characterizing symmetric covers of given symmetric graphs is an important approach for studying symmetric graphs, refer to \cite{brouwer1989Distanceregular,du1998arctransitive}.

We say $\Gamma$ is a \textit{faithful $G$-arc-transitive cover} of $\Gamma_\mathcal{B}$ if $\Gamma$ is a cover of $\Gamma_\mathcal{B}$ and $G\leqs\Aut\Gamma$ is faithful on $\mathcal{B}$, that is, $G\leqs\Aut\Gamma_\mathcal{B}$.
We prove that each symmetric cover is a composition of a faithful cover and a normal cover, see Lemma~\ref{lem:faith}.
Symmetric covers of complete graphs induced by certain covering transformation groups have been obtained by \cite{du1998arctransitive,du20052arctransitive}.
The first main theorem of this paper, Theorem~\ref{thm:completecover}, characterizes faithful $G$-arc-transitive covers of $\K_n$, see Section~\ref{sec:cover}.

\begin{theorem}\label{thm:completecover}
	Let $\Sigma=\K_n$ with vertex set $\Omega$, $G\leqslant\Aut\Sigma=\S_n$, and let $\omega,\omega'\in\Omega$ be two distinct vertices.
	Then a graph $\Gamma$ is a non-trivial connected faithful $G$-arc-transitive cover of $\Sigma$ if and only if $\Gamma=\Cos(G,L,LgL)$ for some transitive proper subgroup $L$ of $G_\omega$ on $\Omega\setminus\{\omega\}$ and some element $g\in \mathbf{N}_G(L_{\omega'})$ such that $(\omega,\omega')^g=(\omega',\omega)$, $g^2\in L$ and $G=\langle L,g\rangle$.
\end{theorem}

This provides a generic construction of faithful symmetric covers of complete graphs.
For instance, taking $G=\S_n$, $g=(\omega\omega')$ and $L$ to be a regular cyclic subgroup $\mathbb{Z}_{n-1}$ of $\S_{n-1}$, the coset graph $\Gamma=\Cos(G,L,LgL)$ is a largest $G$-arc-transitive cover of $\K_n$ with $G\leqslant\Aut\K_n$ in the sense that the number of vertices is the largest.
This particularly has the following consequence.

\begin{corollary}
	For each integer $n\geqslant 4$, a complete graph $\K_n$ has arc-regular covers with $(n-2)!n$ vertices.
\end{corollary}

As noticed above, if $\Gamma$ is a cover of $\Gamma_\mathcal{B}$, then $\Gamma$ and $\Gamma_\mathcal{B}$ have the same valency.
The converse of the above statement is not generally true, and we say $\Gamma$ is a \textit{pseudocover} of $\Gamma_\mathcal{B}$ if $\val(\Gamma)=\val(\Gamma_\mathcal{B})$ and $\Gamma$ is not a cover of $\Gamma_\mathcal{B}$.
The first explicit example of pseudocover was discovered by Praeger, Li and Zhou in~\cite{li2010Imprimitive}.
Li and the author gave a criterion for deciding an extender to be a pseudocover in \cite{Li2022cover}, see Theorem~\ref{thm:pscover}.
The criterion gives a method for constructing symmetric pseudocovers of complete graphs with non-prime valencies.
Construction~\ref{cons:kab} leads to the following theorem.

\begin{theorem}\label{thm:complete}
	A complete graph $\K_n$ with $n\geq 3$ has a connected symmetric pseudocover if and only if $n-1$ is not a prime.
\end{theorem}

We remark that if a $G$-arc-transitive graph $\Gamma$ has a prime valency, then $G_\beta$ is primitive on $\Gamma(\beta)$ for vertex $\beta$ in $\Gamma$.
Note that a connected graph $\Gamma$ is a $G$-arc-transitive cover of its quotient $\Gamma_\mathcal{B}$ if and only if $\Gamma$ and $\Gamma_\mathcal{B}$ have the same valency, see Section~\ref{sec:prel}.

Theorem~\ref{thm:completecover} provides a tool to classify faithful $G$-arc-transitive covers of $\K_n$ with valency $n-1$ for $G\leqslant\Aut\K_n=\S_n$.
We give explicit classifications of $G$-arc-transitive covers and pseudocovers of $\K_n$ with $G\leqslant\Aut\K_n=\S_n$ for $n=4$ and $5$, see Corollary~\ref{coro:exk4k5}.
However, it is hard to determine all such covers and pseudocovers of complete graphs $\K_n$ for large integers $n$, which leads us to propose the following problem.

\begin{problem}\label{problem:cover}
	Let $\Sigma=\K_n$ with $n\geqslant 6$.
	Classify connected faithful $G$-arc-transitive covers and pseudocovers of $\Sigma$ for each $2$-transitive permutation group $G\leqslant \S_n$.
\end{problem}

\section{Coset graph representation and quotient graphs}\label{sec:prel}

The \textit{coset graph representation} is an important notion for studying symmetric graphs and is the main method of the constructions of symmetric graphs in this paper.
For an abstract group $G$, a subgroup $H<G$ and an element $g\in G\setminus H$ such that $g^2\in H$, one can define a graph $\Gamma=(V,E)$, called a \textit{coset graph} of $G$ and denoted by $\Cos(G,H,HgH)$, such that
\[\begin{array}{l}
	V=[G:H]=\{Hx\mid x\in G\},\\
	E=\bigl\{\{Hx,Hy\}\mid yx^{-1}\in HgH\bigr\}.
\end{array}\]
Then $G$ induces a group of automorphisms of $\Gamma$ by the coset action of $G$ on $[G:H]$.
We state some basic properties of coset graphs blow, the reader may refer to~\cite[Section 2]{Li2022cover}.

\begin{proposition}\label{prop:coset-graph}
	Suppose that $\Sigma$ is a $G$-arc-transitive graph with $G\leqs\Aut\Sigma$ and $\omega$ is a vertex in $\Sigma$.
	Then $\Sigma=\Cos(G,H,HgH)$, where $H=G_\omega$ and $g^2 \in H$ such that
	\begin{enumerate}
		\item $\Sigma$ has valency equal to $\bigl|H:H\cap H^g\bigr|$;
		\item the action of $H=G_\omega$ on the neighborhood $\Sigma(\omega)$ is equivalent to $H$ acting on $[H:H\cap H^g]$ by right multiplication;
		\item $\Sigma$ is connected if and only if $G=\langle H,g\rangle$.
	\end{enumerate}
\end{proposition}

For complete graph $\Sigma=\K_n$ with vertex set $\Omega$, the automorphism group of $\Sigma$ is the symmetric group $\S_n$ on $\Omega$.
It is clear that $\Sigma$ is symmetric and $\Aut\Sigma$ acts transitively on $2$-arcs of $\Sigma$.
Note that each ordered pair $(\omega,\omega')$ in $\Omega$ with $\omega\neq\omega'$ is an arc of $\Sigma$.
Thus if $G\leqs\Aut\Sigma$ acts transitively on arcs of $\Sigma$, then $G$ is a $2$-transitive subgroup of $\Aut\Sigma=\S_n$.
\begin{proposition}\label{prop:complete}
	Suppose that $\Sigma=\K_n$ is a complete graph for $n\geqs 3$ with vertex set $\Omega$.
	Then $\Sigma=\Cos(G,H,HgH)$ where $G\leqslant \S_n$ is a $2$-transitive group on $\Omega$, $H=G_\omega$ for $\omega\in\Omega$, $\omega^g\neq \omega$ and $\omega^{g^2}=\omega$.

	Conversely, a coset graph $\Sigma=\Cos(G,H,HgH)$ is a complete graph if and only if $G$ acts $2$-transitively on $[G:H]$.
\end{proposition}

By definition, a $G$-arc-transitive extender of $\Sigma$ is a $G$-arc-transitive graph, and so can be presented as a coset graph.
The following proposition identifies such coset graphs.

\begin{proposition}\label{prop:extender}
	Let $\Gamma$ and $\Sigma$ be symmetric graphs.
	Then $\Gamma$ is a $G$-arc-transitive extender of $\Sigma$ if and only if $\Gamma=\Cos(G,L,LgL)$ and $\Sigma=\Cos(G,H,HgH)$, where $L<H$ and $g\in G\setminus L$ such that $g^2\in L$.
\end{proposition}

Suppose that $\Gamma$ is a $G$-arc-transitive extender of $\Gamma_\mathcal{B}$ with respect to block system $\mathcal{B}$ of $G$.
Let $(\beta,\gamma)$ is an arc in $\Gamma$ and $\beta\in B$ for $B\in\mathcal{B}$.
Observe that the blocks with some vertices adjacent to $\beta$, admit a block system of $\Gamma(\beta)$.
Hence if $\Gamma$ is \textit{$G$-locally-primitive}, that is, $G_\beta$ is primitive on $\Gamma(\beta)$, then $\Gamma$ is a pseudocover of $\Gamma_\mathcal{B}$ only when $\Gamma$ is not connected.
In fact, each symmetric graph has a disconnected symmetric pseudocover (see~\cite[Proposition 2.6]{Li2022cover} and Figure~\ref{fig:disconnected}).
\begin{figure}[!ht]
	\begin{center}
		\begin{tabular}{cc}
			\includegraphics[scale=0.8]{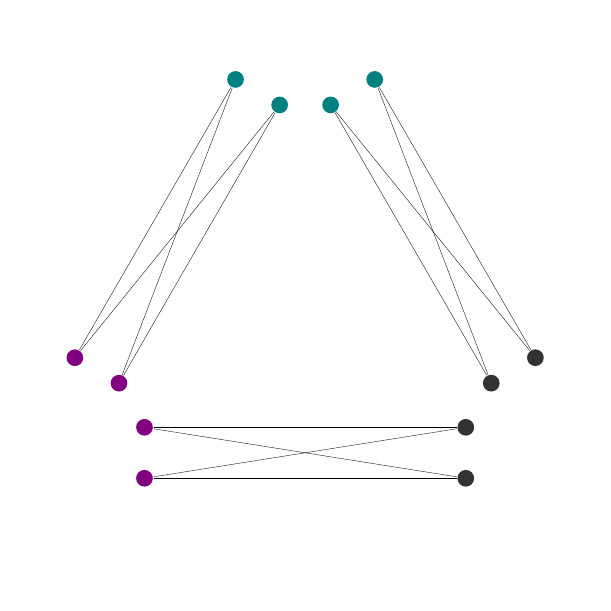}
			&
			\includegraphics[scale=0.8]{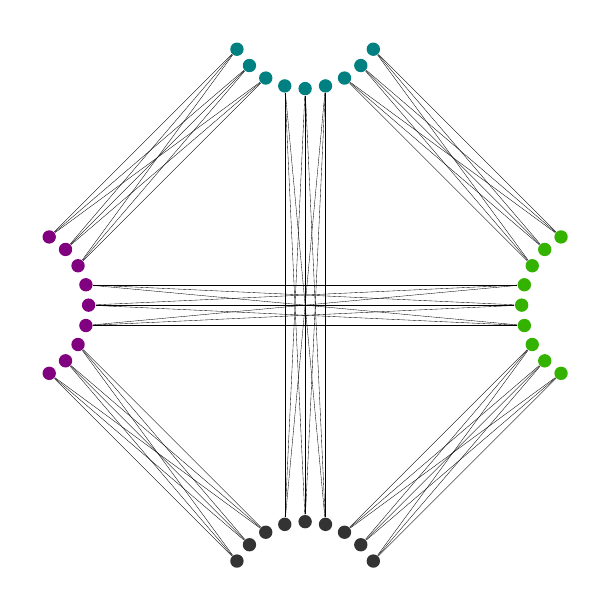}
		\end{tabular}
	\end{center}
	\caption{Disconnected pseudocovers of $\K_3$ and $\K_4$}\label{fig:disconnected}
\end{figure}

Li and the author provide a criterion for deciding an extender to be a cover or a pseudocover in~\cite{Li2022cover} and we state it blow.

\begin{theorem}\label{thm:pscover}
	Let $\Gamma$ be a $G$-arc-transitive extender of $\Gamma_\mathcal{B}$ such that $\val(\Gamma)=\val(\Gamma_{\mathcal{B}})$.
	Let $(\alpha,\beta)$ be an arc of $\Gamma$ and let $A,B\in\mathcal{B}$ such that $\alpha\in A$ and $\beta\in B$.
	Then the followings are equivalent:
	\begin{enumerate}
		\item $\Gamma$ is a $G$-arc-transitive pseudocover of $\Gamma_{\mathcal{B}}$;

		\item $G_\alpha$ is an intransitive subgroup of $G_{A}$ on $\Gamma_\mathcal{B}(A)$;

		\item $G_{\alpha B}=G_{\alpha}\cap G_B$, $G_{A\beta}=G_{A}\cap G_\beta$ and $G_{\alpha\beta}$ are pairwise different.
	\end{enumerate}
\end{theorem}

A typical case is that $\mathcal{B}$ is a block system consisting of orbits of a normal subgroup $N$ of $G$.
In this case, we say the quotient graph $\Gamma_\mathcal{B}$ is a \textit{normal quotient} of $\Gamma$, denoted by $\Gamma_N$.
In this case, $\Gamma$ is a multicover of $\Gamma_N$ and $\val(\Gamma_N)$ divides $\val(\Gamma)$.
In particular, we say $\Gamma$ is a \textit{normal cover} of $\Gamma_N$ if $\Gamma$ is a cover of $\Gamma_N$, equivalently $\val(\Gamma)=\val(\Gamma_N)$.
The following lemma shows that a symmetric cover is a normal cover of a faithful cover.

\begin{lemma}\label{lem:faith}
    Suppose that $\Gamma$ is a $G$-arc-transitive cover of $\Sigma=\Gamma_\mathcal{B}$.
    Let $N$ be the kernel of $G$ acting on $\mathcal{B}$.
    Then $\Gamma$ is a normal cover of $\Gamma_N$ and $\Gamma_N$ is a faithful $G/N$-arc-transitive cover of $\Sigma$.
\end{lemma}
\begin{proof}
    Suppose that $(\alpha,\beta)$ is an arc in $\Gamma$ and $A,B\in\mathcal{B}$ such that $\alpha\in A$ and $\beta\in B$.
	Let $\overline{\alpha}$ and $\overline{\beta}$ be vertices in $\Gamma_N$ corresponding to blocks (induced by $N$) containing $\alpha$ and $\beta$ respectively.
	Then $(A,B)$ is an arc in $\Sigma$ and $(\overline{\alpha},\overline{\beta})$ is an arc in $\Gamma_N$.
	Notice that $G_{\bar{\beta}}=NG_{\beta}\leqs G_B$, then $G_{\bar{\alpha}\bar{\beta}}\leqs G_{\overline{\alpha}B}$, and hence
	\[\val(\Gamma_N)=|G_{\overline{\alpha}}:G_{\bar{\alpha}\bar{\beta}}|\geqs |G_{\overline{\alpha}}:G_{\overline{\alpha}B}|.\]
	By equivalence between part~(i) and (ii) of Theorem~\ref{thm:pscover}, $G_{\alpha}$ is transitive on $\Sigma(A)$.
	Then $G_{\overline{\alpha}}=G_{\alpha} N$ has an orbit $\Sigma(A)$ of the action on $\mathcal{B}=V\Sigma$ with kernel $N$ and a stabilizer $G_{\overline{\alpha}B}$. 
	Thus we obtain that 
	\[\val(\Gamma_N)\geqs |G_{\overline{\alpha}}:G_{\overline{\alpha}B}|=|\Sigma(A)|=\val(\Sigma).\]
	Since $\Gamma$ is a cover of $\Sigma$, it follows that $\val(\Gamma_N)\geqs\val(\Sigma)=\val(\Gamma)$.
	Hence $\val(\Gamma_N)=\val(\Gamma)$ as $\val(\Gamma)$ is divisible by $\val(\Gamma_N)$.
	Thus $\Gamma$ is a normal cover of $\Gamma_N$.
	Since $G_{\overline{\alpha}}$ is transitive on $\Sigma(A)$ and $\val(\Gamma_N)=\val(\Sigma)$, $\Gamma_N$ is faithful a cover of $\Sigma$ by Theorem~\ref{thm:pscover}.
	The induced permutation group of $G$ on $\mathcal{B}$ is $G^\mathcal{B}=G/N$, then $\Gamma_N$ is a faithful $G/N$-arc-transitive cover of $\Sigma$. 
\end{proof}

\section{Symmetric covers of complete graphs}\label{sec:cover}\

Suppose that $\Gamma$ is a $G$-arc-transitive extender of complete graph $\Sigma=\K_n$.
With Theorem~\ref{thm:pscover} and Propositions~\ref{prop:extender}, \ref{prop:complete}, we immediately obtain the following lemma.
\begin{lemma}\label{lem:complete}
	Suppose that $\Gamma$ is a $G$-arc-transitive extender of complete graph $\Sigma=\K_n$. 
	Let $\Omega$ be the vertex set of $\Sigma$ and let $\omega,\omega'$ be two vertices in $\Omega$.
	Then $G^\Omega$ is a $2$-transitive subgroup of $\S_n$ and
	\[\Gamma=\Cos(G,L,LgL),\ \ \ \Sigma=\Cos(G,H,HgH),\]
	where $H=G_\omega$, $L<H$ and $(\omega,\omega')^g=(\omega',\omega)$.
\end{lemma}
\begin{proof}
	Since $\Sigma=\K_n$ is a complete graph, $G\leqs\S_n$ is $2$-transitive on $\Omega$ by Proposition~\ref{prop:complete}.
	With Proposition~\ref{prop:extender}, we have $\Gamma=\Cos(G,L,LgL)$ and $\Sigma=\Cos(G,H,HgH)$ with $H$ a stabilizer of some vertex in $\Omega$, $L<H$ and $g^2\in L$.
	Since the vertex stabilizers of $G$ on $\Sigma$ are conjugate, we shall assume that $H=G_\omega$.
	Then  $\omega^g\neq \omega$ and $\omega^{g^2}=\omega$.

	Note that $H$ is transitive on $\Omega\setminus\{\omega\}$, there exists an element $h\in H$ such that $\omega^{g}=\omega'^h$.
	Since $\Cos(G,L,LgL)=\Cos(G,L^g,(LgL)^g)$, we may replace $L$ and $g$ by $L^h$ and $g^h$ respectively.
	Thus $(\omega,\omega')^g=(\omega',\omega)$ in this case, which completes the proof.
\end{proof}

Assume further $\Gamma$ is a faithful symmetric cover of $\Sigma$.
By the equivalence between part~(i) and (ii) of Theorem~\ref{thm:pscover}, $L< G_\omega$ is transitive on $\Sigma(\omega)=\Omega\setminus\{\omega\}$.
Hence We observe that for any element $g\in G$ which does not fix $\omega$, the group $\langle L,g\rangle$ must be $2$-transitive.

Therefore, the above disscussion leads to the following generic construction of covers of complete graphs.

\begin{construction}\label{cons:coversn}
	Let $\Omega$ be a set of size $n$, and $X=\S_n$ the symmetric group on $\Omega$.
	Fix two points $\omega,\omega'\in\Omega$.
	Let $L$ be a transitive subgroup of $X_\omega$ on $\Omega\setminus\{\omega\}$.
	Choose $g_0\in \mathbf{N}_{X_{\omega\omega'}}(L_{\omega'})$ such that $g_0^2\in L$.
	Set $g=(\omega\omega')g_0$, $G=\langle L,g\rangle$ and $H=G_\omega$.
	Define
	\[\Gamma=\Cos(G,L,LgL),\ \ \ \Sigma=\Cos(G,H,HgH)\]
\end{construction}

\begin{lemma}\label{lem:coversn}
	With the notations defined in Construction~\ref{cons:coversn}, we have
	\begin{enumerate}
		\item $\Sigma\cong\K_n$, a $G$-arc-transitive complete graph with $n$ vertices;

		\item $\Gamma$ is a connected $G$-arc-transitive cover of $\Sigma$;
			
		\item If $g_0=1$, then $g=(\omega\omega')$ is a transposition and $G=\S_n$.
	\end{enumerate}
\end{lemma}
\begin{proof}
	By definition, $g_0$ fixes both $\omega$ and $\omega'$, and so the transposition $(\omega\omega')$ and $g_o$ commute.
	Thus $g^2=(\omega\omega')^2g_0^2=g_0^2\in L$, and then coset graph $\Gamma$ is undirected with $\omega^g=\omega'$.

	The group $G$ is $2$-transitive on $\Omega$ since $L\leqslant G_\omega$ is transitive on the neighborhood $\Sigma(\omega)=\Omega\setminus\{\omega\}$ and $g\in G$ interchanges $\omega$ and $\omega'$.
	Hence $\Sigma$ is a complete graph and $G$ is transitive on the arcs of $\Sigma$, as in part~(i).
	
	By assumption $G=\langle L,g\rangle$, the graph $\Gamma$ is a connected $G$-arc-transitive extender of $\Sigma$.
	Observe that $L\cap L^g\leqslant G_{\omega}\cap G_{\omega'}$, then $L\cap L^g\leqslant L_{\omega\omega'}=L_{\omega'}$.
	On the other hand, $L_{\omega'}^g=L_{\omega'}^{(\omega\omega')g_0}=L_{\omega'}^{g_0}=L_{\omega'}$, and hence $L_{\omega'}\leqslant L\cap L^g$.
	Further, since $L$ is a transitive subgroup of $G_\omega$ on $\Sigma(\omega)=\Omega\setminus\{\omega\}$, it follows that $\val(\Gamma)=|L:L\cap L^g|=|L:L_{\omega'}|=n-1=\val(\Sigma)$.
	Thus $\Gamma$ is a cover of $\Sigma$ since $L$ is transitive on $\Omega\setminus\{\omega\}$ and $\val(\Gamma)=\val(\Sigma)$ by the equivalence between part~(i) and (ii) of Theorem~\ref{thm:pscover}, as in part~(ii).

	Finally, if $g=(\omega\omega')$ is a transposition, then $G$ is a $2$-transitive permutation group with a transposition. 
	Hence $G$ is equal to the symmetric group $\S_n$, as in part~(iii).
\end{proof}

We are now ready to prove Theorem~\ref{thm:completecover}
\begin{proof}[Proof of Theorem~\ref{thm:completecover}:]
	Suppose that $\Gamma$ is a connected $G$-arc-transitive cover of $\Sigma=\K_n$ with $G\leqslant\Aut\Sigma=\S_n$.
	We may assume $\Gamma=\Cos(G,L,LgL)$ and $\Sigma=\Cos(G,H,HgH)$ where $H=G_\omega$, $L<H$ and $(\omega,\omega')^g=(\omega',\omega)$. by Lemma~\ref{lem:complete}.
	Since $\Gamma$ is a connected cover of $\Sigma$ and $\Sigma=\K_n$, it follows that $G=\langle L,g\rangle$ is $2$-transitive on $\Omega$, $g^2\in L$ and $L$ is transitive on $\Omega\setminus\{\omega\}$ by part~(ii) of Theorem~\ref{thm:pscover}.
	In particular, $g$ normalizes $L_{\omega'}=L_{\omega\omega'}$.

	Suppose that $\Gamma=\Cos(G,L,LgL)$ where $L< G_\omega$ is transitive on $\Omega\setminus\{\omega\}$ and $g\in \mathbf{N}_G(L_{\omega'})$ such that $(\omega,\omega')^g=(\omega',\omega)$, $g^2\in L$ and $G=\langle L,g\rangle$.
	Let $g_0=(\omega\omega')g$ then $g_0$ fixes $\omega$ and $\omega'$.
	It follows that $g=(\omega\omega')g_0$ and $L_{\omega'}^{g_0}=L_{\omega'}^g=L_{\omega'}$.
	Hence the groups $G,H,L$ and the elements $g_0,g$ satisfies the conditions in Construction~\ref{cons:coversn}.
	By Lemma~\ref{lem:coversn}, $\Sigma=\K_n$ and $\Gamma$ is a connected $G$-arc-transitive cover of $\Sigma$.
\end{proof}

The cube graph is the unique $G$-arc-transitive cover of $\Sigma=\K_4$ for $G\leqslant\S_4$, see Example 2.5 in~\cite{Li2022cover}. 
By Theorem~\ref{thm:completecover}, we can find all $G$-arc-transitive covers of $\K_5$ for $G\leqslant\S_5$.
\begin{example}\label{exam:covers5}
	Let $\Omega=\{1,2,...,5\}$ be the vertex set of $\Sigma=\K_5$, and $G\leqslant\Aut\Sigma=\S_5$ arc-transitive on $\Sigma$.
	Then $G=\S_5$, $\A_5$ or isomorphic to $\mathrm{F}_5\cong\mathbb{Z}_5{:}\mathbb{Z}_4$.
	
	Let $\Gamma=\Cos(G,L,LgL)$ be a proper connected $G$-arc-transitive cover of $\Sigma$ where $L< G_\omega$ transitive on $\Omega\setminus\{\omega\}$ and $g\in G\setminus G_\omega$ with $g^2\in L$ for some $\omega\in \Omega$.
	Note that $L$ is a proper subgroup of $G_\omega$ of order divided by $4$.
	Thus
	\[(G,G_\omega,L)\cong (\A_5,\A_4,\mathbb{Z}_2^2),\ (\S_5,\S_4,\mathbb{Z}_2^2),\ (\S_5,\S_4,\mathbb{Z}_4),\ (\S_5,\S_4,\D_8)\mbox{ or }(\S_5,\S_4,\A_4).\]

	For each case above, we can find all the double cosets $LgL$ such that $g\in G\setminus G_\omega$ with $g^2\in L$.
	With these double cosets, the corresponding coset graphs $\Gamma=\Cos(G,L,LgL)$ are $G$-arc-transitive covers of $\Sigma$.
	By computing in Magma~\cite{magma}, there are $5$ non-isomorphic connected $\S_5$-symmetric covers of $\Sigma=\K_5$ listed in Table~\ref{tab:K5}, and a unique connected $\A_5$-symmetric cover which is isomorphic to the graph listed at column~4 in Table~\ref{tab:K5}.
	\qed
\end{example}
\begin{table}[!ht]
	\centering
	\begin{tabular}{llllll}
	$L$ & $g$ & $|V\Gamma|$ & $\Aut\Gamma$ &normal cover of $\K_5$? \\ \hline
	$\langle (23)(45),(24)(35)\rangle\cong\mathbb{Z}_2^2$& $(12)$ & $30$ & $\S_3\times \S_5$ & true \\
	$\langle (2345)\rangle\cong\mathbb{Z}_4$& $(12)$ &$30$&$\S_2\times \S_5$ & false \\
	$\langle (2345)\rangle\cong\mathbb{Z}_4$& $(12)(34)$ &$30$&$\S_2\times \S_5$ & false \\
	$\langle (2345),(35)\rangle\cong \mathrm{D}_8$& $(12)$ & $15$ & $\S_5$ & false\\
	$\langle (234),(23)(45)\rangle\cong \A_4$& $(12)$ & $10$ & $\mathbb{Z}_2\times \S_5$ & true\\
	\hline
	\end{tabular}
	\caption{$\S_5$-symmetric covers of $\K_5$}\label{tab:K5}
\end{table}

Notice that $\Aut\Gamma$ is generally not equal to $G$, the cube graph has automorphism group isomorphic to $\mathbb{Z}_2\times\S_4$ for instance. 
With Table~\ref{tab:K5}, we can observe that if $n\leqslant 5$ then $\S_n\lhd \Aut\Gamma$ for $n\leqslant 5$.
In fact, $\S_n$ is also generally not a normal subgroup of $\Aut\Gamma$.
\begin{example}
	Let $G=\S_6$, $L=\langle (23456)\rangle\leqslant G_1$ and $g=(12)$.
	Then $\Gamma=\Cos(G,L,LgL)$ is a connected $G$-arc-transitive cover of $\Sigma=\K_6$.
	By computing in Magma, we obtain that $\Aut\Gamma\cong \mathbb{Z}_4.\mathrm{P\Gamma L}(2,9)$ and $G$ is not normal in $X$.
\end{example}

\section{Symmetric pseudocovers of complete graphs}\label{sec:pscover}\

Let $\Sigma=\K_n$ be a $G$-arc-transitive complete graph of order $n$ with vertex set $\Omega$ and $G\leqslant \S_n$ a $2$-transitive permutation group.
A $G$-arc-transitive extender $\Gamma$ of $\Sigma$ is of form $\Cos(G,L,LgL)$ with $L<G_\omega$ for some $\omega\in\Omega$ and $g\in G$ such that $g^2\in L$ and $G=\langle L,g\rangle$.
Assume further that $\Gamma$ is a pseudocover of $\Sigma$.
Then
\[|L:L\cap L^g|=n-1,\ \mbox{and $L$ is not transitive on the neighborhood $\Sigma(\omega)=\Omega\setminus\{\omega\}$.}\]

The following lemma is a corollary of Theorem~\ref{thm:pscover}, which we shall apply to construct pseudocovers of complete graphs.

\begin{lemma}\label{cons:(L,g)}
	Let $L<\Sym(\Omega\setminus \{\omega\})\cong \S_{n-1}$ and $g\in\Sym(\Omega)=\S_n$ such that
	\begin{enumerate}
		\item $L$ is an intransitive group of order $n-1$ on $\Omega\setminus\{\omega\}$, and
		\item $g$ is an involution such that $L\cap L^g=1$, and $G:=\langle L,g\rangle$ is a $2$-transitive subgroup of $\S_n$.
	\end{enumerate}
	Then $G$ is symmetric on $\Sigma=\K_n$, and $\Cos(G,L,LgL)$ is a connected $G$-arc-transitive pseudocover of $\Sigma$.
\end{lemma}

Recall a well-known result in permutation group theory given by Jones, which is an extension of Jordan's Theorem, see~\cite[Theorem 1.2]{jones2014Primitive}.

\begin{lemma}\label{lem:jones}
	Let $G$ be primitive permutation group of degree $n$. If $G$ contains a cycle fixing at least $3$ points then $\A_n\lhd G$.
\end{lemma}

We will present two generic constructions of pseudocovers of complete graphs $\K_n$.
In the first construction, we take $L$ to be a dihedral group in the case where $n$ is odd.

\begin{construction}\label{cons:k2m}
	Let $\Omega=\{1,...,2m+1\}$ with $m\geqslant 2$, and let
	\[\left\{\begin{aligned}
		a&=(2,3,\cdots,2m+1);\\
		b&=(3,2m+1)(4,2m)\cdots(m+1,m+3);\\
		g&=(1,2)(3,4).
	\end{aligned}\right.\]
	Let $L=\bigl<a^2,b\bigr>\cong \mathrm{D}_{2m}$, and let $G=\left<L,g\right>$.
\end{construction}

\begin{lemma}\label{lem:consdih}
	Let $G$, $L$ and $g$ be constructed in Construction~$\ref{cons:k2m}$. Then $G$ is $2$-transitive on $\Omega$ and $L\cap L^g=1$. Furthermore,
	\[G= \left\{\begin{aligned}
		&\S_{2m+1},\mbox{ $m$ is even};\\
		&\A_{2m+1},\mbox{ $m\geqslant 5$ is odd};\\
		&\mathrm{GL}(3,2), \mbox{ $m=3$}.
	\end{aligned}\right.\]
\end{lemma}
\begin{proof}
	By definitions, $L$ fixes the point $1$ and $1^g=2$.
	Thus $L\cap L^g$ fixes both points $1$ and $2$, and hence $L\cap L^g\leqslant \langle b\rangle\cong\mathbb{Z}_2$.
	Since $[b,g]\neq 1$, we obtain that $L\cap L^g=1$.

	Observe that the subgroup $L$ acting on $\Omega$ has exactly three orbits $\{1\}$, $O_1$ and $O_2$, where
	\[\mbox{$O_1=\{3,...,2m+1\}$, and $O_2=\{2,4,...,2m\}$.}\]
	Since $1^g=2$ and $3^g=4$, it follows that $G=\langle L,g\rangle$ is transitive on $\Omega$.

	Now, let $x=gbg\in G$.
	Then $x$ fixes the point 1 since $1^x=1^{gbg}=2^{bg}=2^g=1$, and $x$ sends the point $3\in O_1$ to the point $2m\in O_2$, since
	\[3^x=3^{gbg}=4^{bg}=(2m)^g=2m.\]
	Hence the subgroup $\langle L,x\rangle$ fixes the point 1 and is transitive on $O_1\cup O_2=\Omega\setminus\{1\}$, and so $G$ is $2$-transitive on $\Omega$.

	Next, we determine the $2$-transitive group $G$.
	For the smallest case $m=2$, we have $a^2=(24)(35)$ and $b=(35)$, and it is then easily shown that
	$G=\langle a^2,b,g\rangle=\langle(24)(35),(35),(12)(34)\rangle=\S_5$.
		
	Assume that $m=3$.
	Then $|\Omega|=7$, $a^2=(246)(357)$, $b=(37)(46)$ and $g=(12)(34)$.
	Computation with Magma shows that there exists an isomorphism $\varphi$ between the group $G=\langle a^2,b,g\rangle$ and $\GL(3,2)$ such that
	\[\varphi(a^2)=\begin{pmatrix}
		1&  0&0  \\
		0&  0& 1 \\
		0& 1 & 1
	\end{pmatrix},\ \
	\varphi(b)=\begin{pmatrix}
		1&  0&0  \\
		0&  1& 0 \\
		0& 1 & 1
	\end{pmatrix},\ \
	\varphi(g)=\begin{pmatrix}
		0&  1&0 \\
		1&  0& 0 \\
		1& 1& 1
	\end{pmatrix}.\]

	Finally, assume that $m\geqslant 4$.
	Let $c=a^2b=b^{a^{-1}}=(2,2m)(3,2m-1)\cdots(m,m+2)$, and $y=[c,g]$.
	Then $c,y\in G=\langle L,g\rangle$.
	Set $c'=(2,2m)(3,2m-1)(4,2m-2)$, we have $[c,c']=1$ and $\mathrm{Supp}(cc')\cap \mathrm{Supp}(g)=1$.
	Hence $[cc',g]=1$ and
	\[y=[c,g]=[cc'c',g]=[cc',g]^{c'}[c',g]=[c',g].\]
	Notice that
	\[\mathrm{Supp}(y)=\mathrm{Supp}([c',g])\subseteq \mathrm{Supp}(c')\cup \mathrm{Supp}(g)=\{1,2,3,4,2m,2m-1,2m-2\}\]
	Moreover,
	\[\left\{\begin{aligned}
		&1^y=1^{c'gc'g}=1^{gc'g}=2^{c'g}=(2m)^g=2m;\\
		&2^y=2^{c'gc'g}=(2m)^{gc'g}=(2m)^{c'g}=(2)^g=1;\\
		&3^y=3^{c'gc'g}=(2m-1)^{gc'g}=(2m-1)^{c'g}=3^g=4;\\
		&4^y=4^{c'gc'g}=(2m-2)^{gc'g}=(2m-2)^{c'g}=4^g=3;\\
		&(2m)^y=(2m)^{c'gc'g}=2^{gc'g}=1^{c'g}=1^g=2;\\
		&(2m-1)^y=(2m-1)^{c'gc'g}=3^{gc'g}=4^{c'g}=(2m-2)^g=2m-2;\\
		&(2m-2)^y=(2m-2)^{c'gc'g}=4^{gc'g}=3^{c'g}=(2m-1)^g=2m-1.
	\end{aligned}\right.
	\]
	Therefore, $y=[a^2b,g]=(1,2m,2)(3,4)(2m-1,2m-2)$ and $y^2$ is a $3$-cycle in $G$.
	Hence $G$ is $2$-transitive and $y^2$ fixes at least $3$ elements, and so $\A_{2m+1}\lhd G$ by Lemma~\ref{lem:jones}.
	Noticing that $a^2,b,g$ are even permutations for odd $m$, and $b$ is an odd permutation for even $m$, we conclude that $G= \A_{2m+1}$ if $m\geqslant 5$ is odd, and $G= \S_{2m+1}$ if $m$ is even.
\end{proof}

By Lemma~\ref{cons:(L,g)}, Each triple $(G,L,g)$ given in Construction~\ref{cons:k2m} gives rise to a pseudocover of $\K_n$.

\begin{lemma}\label{coro:consdih}
Let $G$, $L$ and $g$ be as defined in Construction~$\ref{cons:k2m}$.
Then $\Gamma=\Cos(G,L,LgL)$ is a connected $G$-arc-transitive pseudocover of $\Sigma=\K_n$ with $n=2m+1$.
\end{lemma}

Since a graph with a prime valency has no connected symmetric pseudocovers, it follows that $n=5$ is the smallest integer such that $\K_n$ has a connected symmetric pseudocover.
In fact, the graph $\Gamma$ defined in Lemma~\ref{lem:consdih} with $m=2$ is the unique $\S_5$-arc-transitive pseudocover of $\K_5$.
\begin{example}\label{ex:K_5}
	{\rm
	Let $\Sigma=\K_5$, with vertex set $\Omega=\{1,2,3,4,5\}$, and let $G\leqslant\Sym(\Omega)=\S_5$ be symmetric on $\Sigma$.
	Then $(G,G_\omega)=(\A_5,\A_4)$ or $(\S_5,\S_4)$.

	Assume that $\Gamma=\Cos(G,L,LgL)$ is a connected $G$-arc-transitive pseudocover of $\Sigma$.
	Then we may choose $L<G_\omega$ such that $L$ is intransitive on $\Omega\setminus\{\omega\}$.
	It follows that $(G,G_\omega,L)\cong (\S_5,\S_4,\mathbb{Z}_2^2)$.
	Without loss of generality, we may suppose that $\omega=q$, and $G_\omega=\Sym\{2,3,4,5\}$.
	It is easily shown that $L$ is conjugate in $G_\omega$ to $\langle(24)(35),(35)\rangle\cong\D_4$.

	Let $g\in G=\S_5$ be such that $g^2\in L$ and $\langle L,g\rangle=G$.
	If $|g|>2$, then $1<L\cap L^g\lhd\langle L,g\rangle=\S_5$, which is not possible.
	Thus $g$ is an involution, and it follows since $\langle L,g\rangle=G$ that $g$ is conjugate under $\N_G(L)$ to $(12)(34)$, where $\N_G(L)=\langle(2345),(35)\rangle\cong\D_8$.
	Therefore, up to isomorphism, the graph $\Gamma=\Cos(G,L,LgL)$ is the only connected faithful symmetric pseudocover of $\K_5$, see the first graph in Figure~\ref{fig:psK5}.
	\qed
	}
\end{example}
\begin{figure}[!ht]
	\centering
	\includegraphics[scale=0.3]{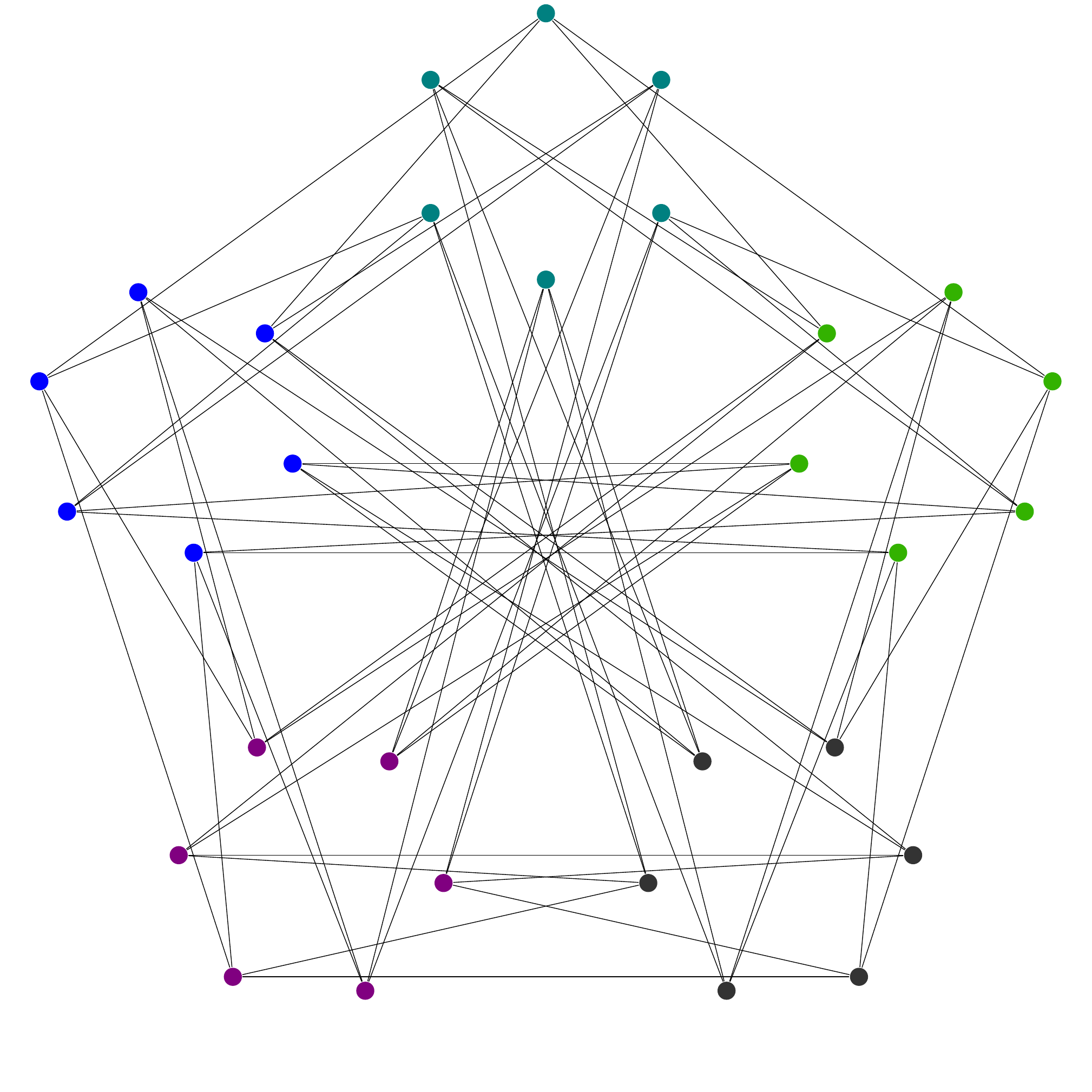}
	\caption{Pseudocover of $\K_5$}\label{fig:psK5}
\end{figure}

Hence we conclude the classification of $G$-arc-transitive covers and pseudocovers of $\K_n$ with $G\leqslant\S_n$ for $n=4$ and $5$, which answers Question~\ref{problem:cover} for $n\leqslant 5$.
\begin{corollary}\label{coro:exk4k5}
	\begin{enumerate}
		\item $\K_4$ has no connected symmetric pseudocovers. The cube graph is the unique $G$-arc-transitive cover of $\K_4$ for $G\leqslant\Aut\K_4=\S_4$.
		\item $\K_5$ has a unique $G$-arc-transitive pseudocover and $5$ non-isomorphic $G$-arc-transitive covers for $G\leqslant\Aut\K_5=\S_5$.
	\end{enumerate}
\end{corollary}
\begin{proof}
	The part~(i) is discussed in Example 2.5 of~\cite{Li2022cover}.
	The part~(ii) is given by Example~\ref{exam:covers5} and Example~\ref{ex:K_5}.
\end{proof}

Our next construction provides examples of connected symmetric pseudocovers of $\K_n$ for all admissible values of $n$, namely, $n-1$ is not a prime.
We shall take an abelian subgroup $L$ of $\Sym(\Omega)$ such that $L\cong\mathbb{Z}_a\times\mathbb{Z}_b$ with $n-1=ab$, and $L$ fixes a point $\omega=n$ and is intransitive on $\Omega\setminus\{\omega\}$.
Then find an involution $g\in\Sym(\Omega)$ such that $\langle L,g\rangle$ is 2-transitive on $\Omega$.
We notice that such an involution $g$ should link the orbits of $L$ together so that $\langle L,g\rangle$ is transitive on $\Omega$.

\begin{construction}\label{cons:kab}
	Suppose that $n=ab+1$ with $1<a\leqslant b$ and $\Omega=\{1,...,n\}$. Define
	\[\left\{\begin{aligned}
		x&=(1,2,\cdots,b)(b+1,\cdots,2b)\cdots\bigl((a-2)b+1,\cdots,(a-1)b\bigr),&\mbox{ product of $a-1$ disjoint $b$-cycles};&\\
		y&=\bigl((a-1)b+1,\cdots,(a-1)b+a\bigr),&\mbox{ an $a$-cycle};&\\
		g_1&=(b,b+1)(2b,2b+1)\cdots\bigl((a-1)b,(a-1)b+1\bigr),&\mbox{ product of transpositions};&\\
		g_2&=\bigl(1,(a-1)b+a+1\bigr)\bigl(2,(a-1)b+a+2\bigr)\cdots(b-a+1,ab+1),&\mbox{ product of transpositions}.&
	\end{aligned}\right.\]
	Let $g=g_1g_2$, $L=\langle x,y\rangle$ and $G=\langle L,g\rangle$.
\end{construction}

We shall prove that $G=\Sym(\Omega)$ or $\Alt(\Omega)$ by a series of lemmas.
The first one shows that $G$ is transitive on $\Omega$.

\begin{lemma}\label{lem:kabprop}
Let $G,L$ and $g$ be as defined in Construction~$\ref{cons:kab}$.
Then $g^2=1$, $L\cong\mathbb{Z}_a\times\mathbb{Z}_b$, $L\cap L^g=1$, and $G=\langle L,g\rangle$ is transitive on $\Omega$.
\end{lemma}
\begin{proof}
	Observe that the support set of $g_1$ is $\{kb\mid 1\leqslant k\leqslant a-1\}\cup\{kb+1\mid 1\leqslant k\leqslant a-1\}$, and all of these points are fixed by $g_2$.
	Hence $[g_1,g_2]=1$ and $g^2=(g_1g_2)^2=g_1^2g_2^2=1$ since $g_1$ and $g_2$ are involutions by definition.

	Note that $y=\bigl((a-1)b+1,\cdots,(a-1)b+a\bigr)$ is an $a$-cycle.
	By definition, $x$ fixes points $(a-1)b+1,...,(a-1)b+a$.
	Hence $[x,y]=1$ and $L=\langle x,y\rangle=\langle x\rangle\times\langle y\rangle\cong\mathbb{Z}_a\times\mathbb{Z}_b$.
	
	Note that both $x$ and $y$ fix the point $n=ab+1$, then $L=\langle x,y\rangle\leqslant G_n$.
	Since $n^g=b-a+1$, it follows that $L\cap L^g\leqslant L$ fixes the point $b-a+1$.
	As $x^k$ does not fix point $b-a+1$ for $1\leqslant k\leqslant b-1$, we have $L\cap L^g\leqslant L_{b-a+1}=\langle y\rangle\cong\mathbb{Z}_2$.
	Note that $y$ fixes the point $(a-1)b$ and $y^g$ maps the point $(a-1)b$ to the point $(a-1)b+2$.
	Thus $[y,g]\neq 1$, and therefore $L\cap L^g=1$.

	Let
	\[\left\{\begin{aligned}
		O_{x,k}&=\{(k-1)b+1,(k-1)b+2,\cdots kb\}\mbox{ for $1\leqslant k\leqslant a-1$};\\
		O_y&=\{(a-1)b+1,\cdots,(a-1)b+a\};\\
		O&=\{(a-1)b+a+1,(a-1)b+a+2,...,ab+1=n\}.
	\end{aligned}\right.\]
	Then each $O_{x,k}$ is an orbit of $\langle x\rangle$, $O_y$ is an orbit of $\langle y\rangle$, and elements in $O$ are fixed by both $x$ and $y$.
	Thus $O_{x,k}$ and $O_y$ are orbits of $L=\langle x,y\rangle$, and elements in $O$ are fixed by $L$.
	Notice that $\Omega=O_{x,1}\cup\cdots\cup O_{x,a-1}\cup O_y\cup O$.
	To show the group $G=\langle L,g\rangle$ is transitive on $\Omega$, it suffices to show $g$ maps each element in $O$ to some element in $O_y$, links subsets $O_{x,1},...,O_{x,a-1}, O_y$.
	By the definition, $g$ maps points $kb\in O_{x,k}$ for $1\leqslant k\leqslant a-2$ to points $kb+1\in O_{x,k+1}$;
	maps the point $(a-1)b\in O_{x,a-1}$ to the point $(a-1)b+1\in O_y$;
	and maps points $(a-1)b+a+t\in O$ to points $t\in O_{x,1}$ for $1\leqslant t\leqslant b-a+1$.
	Hence the group $G=\langle L,g\rangle$ is transitive on $\Omega$.
\end{proof}

The next lemma shows that $G$ contains an alternating subgroup on a large subset of $\Omega$.

\begin{lemma}\label{lem:zk}
Given a triple $(G,L,g)$, as defined in Construction~$\ref{cons:kab}$, let $z_k=y^{(gx)^k}$ and $P=\langle L,z_1,...,z_{a-1}\rangle$.
	Then $z_k=\bigl((a-k-1)b+1,(a-1)b+2,\cdots,(a-1)b+a\bigr)$ is an $a$-cycle for $0\leqslant k\leqslant a-1$ and
	\[\mathrm{Alt}\{1,2,...,(a-1)b+a\}\lhd P\leqslant\mathrm{Sym}\{1,2,...,(a-1)b+a\}\cap G_n,\]
namely, the stabilizer of the point $n$ in $G$ contains a subgroup $\mathrm{Alt}\{1,2,...,(a-1)b+a\}\cong\A_{(a-1)b+a}$.
\end{lemma}
\begin{proof}
	Since $y=\bigl((a-1)b+1,\cdots,(a-1)b+a\bigr)$ is an $a$-cycle, it follows that $z_k=y^{(gx)^k}$ is also an $a$-cycle.
	By definitions, elements $x$, $g_1$ and $g_2$ fix points $(a-1)b+s$ for $2\leqslant s\leqslant a$.
	Thus $gx$ fixes points $(a-1)b+s$ for $2\leqslant s\leqslant a$ and
	\[z_k=y^{(gx)^k}=\bigl(((a-1)b+1)^{(gx)^k},(a-1)b+2,\cdots,(a-1)b+a\bigr).\]
	Notice that $\bigl((a-k-1)b+1\bigr)^{gx}=\bigl((a-k-1)b\bigr)^x=(a-k-2)b+1$ for each $0\leqslant k\leqslant a-2$.
	Hence by induction we obtain that $\bigl((a-1)b+1\bigr)^{(gx)^k}=(a-k-1)b+1$ for $0\leqslant k\leqslant a-1$, and therefore
	\[z_k=\bigl((a-k-1)b+1,(a-1)b+2,\cdots,(a-1)b+a\bigr)\mbox{, for each }0\leqslant k\leqslant a-1.\]

	Recall that $O_{x,k}=\{(k-1)b+1,(k-1)b+2,\cdots kb\}$ for $1\leqslant k\leqslant a-1$ and $O_y=\{(a-1)b+1,\cdots,(a-1)b+a\}$ are orbits of $L$, and $O=\{(a-1)b+a+1,(a-1)b+a+2,...,ab+1=n\}$ is the set of fixed points of $L$.
	Since $z_k$ fixes points in $O$ for each $0\leqslant k\leqslant a-1$, it follows that $P\leqslant \mathrm{Sym}\{1,2,...,(a-1)b+a\}$.

	For each $1\leqslant k\leqslant a-1$, the element $z_{a-k}$ maps the point $(k-1)b+1\in O_{x,k}$  to the point $(a-1)b+2\in O_y$.
	Thus $P=\langle L,z_1,...,z_{a-1}\rangle$ is a transitive subgroup of $\mathrm{Sym}\{1,2,...,(a-1)b+a\}$.

	By definition, the element $x$ fixes points $(a-1)b+2$,..., $(a-1)b+a$,
	and then
	\[\begin{aligned}
		z_k^{x^j}&=\bigl((a-k-1)b+1,(a-1)b+2,\cdots,(a-1)b+a\bigr)^{x^j}\\
		&=\bigl(((a-k-1)b+1)^{x^j},(a-1)b+2,\cdots,(a-1)b+a\bigr).
	\end{aligned}\]
	Note that $x^j$ maps the point $(a-k-1)b+1$ to the point $(a-k-1)b+1+j$ for each $1\leqslant k\leqslant a-1$ and $0\leqslant j\leqslant b-1$.
	Hence
	\[z_k^{x^j}=\bigl((a-k-1)b+1+j,(a-1)b+2,\cdots,(a-1)b+a\bigr)\mbox{, for $1\leqslant k\leqslant a-1$ and $0\leqslant j\leqslant b-1$.}\]
	Let $p_\ell=\bigl(\ell,(a-1)b+2,...,(a-1)b+a\bigr)$ for $1\leqslant \ell\leqslant (a-1)b+1$, then $p_\ell=z_k^{x^j}\in G$ for some $k$ and $j$.
	Thus the subgroup $\langle p_\ell\mid 2\leqslant \ell\leqslant (a-1)b+1\rangle\leqslant P$ is transitive on $\{2,3,...,(a-1)b+a\}$ and fixes the point $1$.
	Hence $P$ is a $2$-transitive subgroup of $\mathrm{Sym}\{1,2,...,(a-1)b+a\}$.

	Finally, in the case where $b=2$, we have that $a=2$, and $n=ab+1=5$.
	Thus $x=(12)$, $z_0=y=(34)$ and $z_1=(14)$, and hence $P=\langle x,y,z_1\rangle=\mathrm{Sym}\{1,2,3,4\}$.
	Assume that $b\geqslant 3$.
	Noticing that $y\in P$ is an $a$-cycle fixes elements $1,...,b$, and $P$ is $2$-transitive on $\{1,2,...,(a-1)b+a\}$, by Lemma~\ref{lem:jones}, we conclude that $\mathrm{Alt}\{1,2,...,(a-1)b+a\}\lhd P\leqslant G_n$.
\end{proof}

We now determine the structure of $G=\langle L,g\rangle$.

\begin{lemma}\label{lem:kabAn}
A group $G=\langle L,g\rangle$ defined in Construction~$\ref{cons:kab}$ is $2$-transitive on $\Omega$, and further,
	\[G= \left\{\begin{aligned}
		&\A_{n},\mbox{ $a$ is odd and $b$ is even};\\
		&\S_{n},\mbox{ otherwise}.
	\end{aligned}\right.\]
\end{lemma}
\begin{proof}
	Let $P=\langle L,z_1,...,z_{a-1}\rangle$ be the subgroup of $G$ defined in Lemma~\ref{lem:zk}, so that $\mathrm{Alt}\{1,2,...,(a-1)b+a\}\lhd P\leqslant G_n$.
	Since $G$ is transitive by Lemma~\ref{lem:kabprop}, to prove $G$ is $2$-transitive, we only need to find elements $q_t\in G$ for $1\leqslant t\leqslant b-a$ such that $q_t$ fixes the point $n$ and maps points $(a-1)b+a+t$ into subset $\{1,2,...,(a-1)b+a\}$.

	Assume that $a=2$.
	Then $y=(b+1,b+2)$ is a $2$-cycle in $L$, and $L\leqslant P=\mathrm{Sym}\{1,2,...,b+2\}\leqslant G_n$.
	Thus $(t,b+2)\in P\leqslant G_n$ for $1\leqslant t\leqslant b-a+1$.
	Let $q_t=(t,b+2)^g$.
	Then
	\[q_t=(t,b+2)^{g_1g_2}=\bigl((a-1)b+a+t,b+2\bigr) \mbox{ for $1\leqslant t\leqslant b-a$.}\]
	Hence $q_t$ fixes the point $n$ and maps the point $(a-1)b+a+t$ into the subset $\{1,...,b+2\}$.
	Therefore, $G$ is a $2$-transitive group when $a=2$.
	
	Assume now that $a\geqslant 3$.
	Then $p_t=(t,b-1)\bigl((a-1)b+2,(a-1)b+3\bigr)\in P$ for each $1\leqslant t\leqslant b-a+1$ since $\mathrm{Alt}\bigl\{1,2,...,(a-1)b+a\bigr\}\lhd P< G$.
	By definition, the points $b-1$, $(a-1)b+2$ and $(a-1)b+3$ are fixed by $g$, and $g$ maps the point $t$ to the point $(a-1)b+a+t$.
	Let $q_t=p_t^g$, then we have
	\[\begin{aligned}
		q_t&=p_t^g=\bigl(t^g,(b-1)^g\bigr)\bigl((a-1)b+2,(a-1)b+3\bigr)^g\\
		&=\bigl((a-1)b+a+t,b-1\bigr)\bigl((a-1)b+2,(a-1)b+3\bigr)\mbox{, for $1\leqslant t\leqslant b-a$}.
	\end{aligned}\]
	Thus $q_t$ fixes the point $n$, and maps the point $(a-1)b+a+t$ into subset $\{1,...,(a-1)b+a\}$.
	So $G$ is a $2$-transitive group.
	
	Finally, since $y$ is an $a$-cycle fixing the points ,
	we conclude that $\A_n\lhd G$ by Lemma~\ref{lem:jones}, since $G$ is $2$-transitive $y$ fixes at least $b+1\geqslant 3$ points, in fact, $y$ fixes points $1,...,b$ and $n=ab+1$.
	It is easy to see that $x$, $y$ and $g$ are all even permutations if and only if $a$ is odd and $b$ is even.
	We conclude that $G=\A_n$ if $a$ is odd and $b$ is even, or $G=\S_n$ otherwise.
\end{proof}

Now we are ready to prove Theorem~\ref{thm:completecover}.

\begin{proof}[Proof of Theorem~$\ref{thm:completecover}$:]
	Let $G$ be a group associated with a subgroup $L<G$ and a 2-element $g\in G$ defined in Construction~\ref{cons:kab}.
	Then $G$ is $2$-transitive on $\Omega$ by Lemma~\ref{lem:kabAn}, and so $G$ is symmetric on the complete graph $\Sigma=\K_n$ with vertex set $\Omega$.

	Let $\Gamma=\Cos(G,L,LgL)$.
	Since $L$ fixes the point $n$ and $g^2=1$, it follows that $\Gamma$ is a connected $G$-arc-transitive extender of $\Sigma$.
	Lemma~\ref{lem:kabprop} shows that $|L|=ab=n-1$ and $L\cap L^g=1$, then $\Gamma$ is of valency $ab=n-1$, equal to the valency of $\Sigma=\K_n$.
	Further, as $L$ is intransitive on $\{1,2,...,n-1\}$, we conclude that $\Gamma$ is a connected $G$-arc-transitive pseudocover of $\Sigma=\K_n$ by Lemma~\ref{cons:(L,g)}.
\end{proof}

\bigskip

\noindent{\bf Acknowledgements}
The author is grateful to his supervisor Prof. Cai Heng Li for his useful advice.

\end{document}